\documentclass[12pt]{article}
\usepackage{amsthm}
\usepackage{amsmath}
\usepackage[centerlast]{caption2}
\usepackage{amssymb}
\usepackage{latexsym}
\usepackage{graphicx}

\newtheorem{theorem}{Theorem}
\newtheorem{corollary}{Corollary}
\newtheorem{lemma}{Lemma}
\theoremstyle{remark}

\newtheorem{conjecture}{Conjecture}
\newtheorem{remark}{Remark}

\title{On the  Brannan's  conjecture}
\author{R\'obert Sz\'asz}

\date{}
\begin{document}
\maketitle

\setlength{\textheight}{240mm} \setlength{\textwidth}{135mm}
\addtolength{\topmargin}{-12.5mm}

\def\ds{\displaystyle}
\theoremstyle{definition}
\newline
\footnote{\textbf{AMS \ classification: 30C50}\ }

\begin{abstract}We will prove the Brannan conjecture for particular values of the parameter.  The basic tool of the study is an integral representation  published  in a recent work \cite{3}.
\end{abstract}
\section{Introduction}
We consider the following   Mac-Laurin development
\begin{equation}\label{x9ed7za}
\frac{(1+xz)^\alpha}{(1-z)^\beta}=\sum_{n=0}^\infty{A_n(\alpha,\beta,x)}z^n,
\end{equation}
where $\alpha>0, \ \beta>0, \ x=e^{i\theta}, \ \theta\in[-\pi,\pi],$
and   $z\in{U.}$    It is easily seen, that the radius of convergence of the series  (\ref{x9ed7za})  is equal to $1.$
In \cite{5} the author conjectured, that  if   $\alpha>0, \ \beta>0$  and    $|x|=1,$    then    $$|A_{2n-1}(\alpha,\beta,x)|\leq{A_{2n-1}(\alpha,\beta,1)},$$
where   $n$   is a natural number.
Partial results regarding this question were  already    proved in   \cite{1}, \cite{2}, \cite{5},    \cite{8}.\\
The case $\beta=1,$     $\alpha\in(0,1)$  is still open. Regarding this case   partial     results     were  obtained   in \  \cite{3},  \cite{4}, \cite{6},  \cite{7}.
We will  prove some partial results regarding  the case $\beta=1,$  and   $\alpha\in(0,1).$  In order to do this we will use an integral representation   proved in \ \cite{3},  and     the fact that the conjecture   was   proved for $2n-1\leq51$     in \cite{6}.

\section{Preliminaries}

We inroduce the notation:
$$A_n(\alpha,1,x)=A_n(\alpha,x).$$
It is easily seen that
\begin{eqnarray}\label{n8m9nn}
A_{2n-1}(\alpha,x)=\sum_{k=0}^{2n-1}\frac{(-\alpha)_k(-x)^k}{k!}=1+\frac{\alpha}{1!}x-\frac{\alpha(1-\alpha)}{2!}x^2+\frac{\alpha(1-\alpha)(2-\alpha)}{3!}x^3\\-\frac{\alpha(1-\alpha)(2-\alpha)(3-\alpha)}{4!}x^4+\ldots+
\frac{\alpha(1-\alpha)(2-\alpha)\ldots(2n-2-\alpha)}{(2n-1)!}x^{2n-1}.\nonumber
\end{eqnarray}
We denote $$B(t,\theta)=\frac{\cos\theta+t+t^{2n-1}\cos2n\theta+t^{2n}\cos(2n-1)\theta}{1+t^2+2t\cos\theta},$$
 $$C(t,\theta)=\frac{\sin\theta+t^{2n-1}\sin2n\theta+t^{2n}\sin(2n-1)\theta}{1+t^2+2t\cos\theta}.$$
We need the following lemmas in our study.
\begin{lemma}
For  $\alpha\in(0,1)$  the following equality holds:
\begin{eqnarray}
\Phi(\theta)=-\Gamma(\alpha)\Gamma(-\alpha)A_{2n-1}(\alpha,e^{i\theta})=\int_0^1F(t)\Big[\frac{1}{\alpha}+\sum_{k=1}^{2n-1}(-t)^{k-1}e^{ki\theta}\Big]dt\nonumber\\=\int_0^1F(t)\Big[\frac{1}{\alpha}+B(t,\theta)+iC(t,\theta)\Big]dt,\nonumber
\end{eqnarray}
where  $F'(t)=-t^{-1-\alpha}(1-t)^{\alpha-1},$   and   $F(1)=0.$
\end{lemma}
\begin{proof}
We use two equalities in order to prove the   assertion of the lemma, the  first one is the following:
\begin{eqnarray}\label{1745fd55dw}A_n(\alpha,x)=1+\frac{1}{\Gamma(\alpha)\Gamma(-\alpha)}\int_0^1\Big(\sum_{k=1}^n\frac{(-tx)^k}{k}\Big)t^{-\alpha-1}(1-t)^{\alpha-1}dt,
\end{eqnarray}
which has been deduced in \cite{3},
while the  second    one     is   the  well-known equality  $B(p,q)=\int_0^1t^{p-1}(1-t)^{q-1}dt=\frac{\Gamma(p)\Gamma(q)}{\Gamma(p+q)}.$
Replacing $p=\alpha$      and  $q=1-\alpha,$    we get
$$\int_0^1t^{\alpha-1}(1-t)^{-\alpha}dt=\frac{\Gamma(\alpha)\Gamma(1-\alpha)}{\Gamma(1)}=-\alpha\Gamma(\alpha)\Gamma(-\alpha)=\int_0^1t^{-\alpha}(1-t)^{\alpha-1}dt.$$
This equality and  (\ref{1745fd55dw})  imply  that
\begin{eqnarray}\Gamma(\alpha)\Gamma(-\alpha)A_n(\alpha,x)=\frac{1}{\alpha}\int_0^1tF'(t)dt-\int_0^1\Big(\sum_{k=1}^n\frac{(-tx)^k}{k}\Big){F'(t)}dt,\nonumber
\end{eqnarray}
or equivalently
\begin{eqnarray}\label{1dp3k07}-\Gamma(\alpha)\Gamma(-\alpha)A_n(\alpha,x)=\int_0^1{F'(t)}\Big(\frac{-t}{\alpha}+\sum_{k=1}^n\frac{(-tx)^k}{k}\Big)dt.\nonumber
\end{eqnarray}
Now  integrating by parts and using   that   $\lim_{t\rightarrow0}F(t)\Big(\frac{t}{\alpha}+\sum_{k=1}^n\frac{(-tx)^k}{k}\Big)=0$   and    $\lim_{t\rightarrow1}F(t)\Big(\frac{t}{\alpha}+\sum_{k=1}^n\frac{(-tx)^k}{k}\Big)=0,$   we infer  that
\begin{eqnarray}\label{2ww4dp3k07}-\Gamma(\alpha)\Gamma(-\alpha)A_n(\alpha,x)=\int_0^1{F(t)}\Big(\frac{1}{\alpha}+\sum_{k=1}^n{(-t)^{k-1}}x^k\Big)dt.\nonumber
\end{eqnarray}
  We get the desired equality   replacing    $n$  by $2n-1$.
\end{proof}
\begin{remark}
It is easily seen that the condition  $\alpha\in(0,1)$  implies the existence of the integrals in the previous lemma and its proof.
\end{remark}
We denote
\begin{eqnarray}\label{234n234}
\Phi(\theta)=-\Gamma(\alpha)\Gamma(-\alpha)A_{2n-1}(\alpha,e^{i\theta})=\int_0^1F(t)\Big[\frac{1}{\alpha}\nonumber\\+\frac{\cos\theta+t+t^{2n-1}\cos2n\theta+t^{2n}\cos(2n-1)\theta}{1+t^2+2t\cos\theta}\\+i\frac{\sin\theta+t^{2n-1}\sin2n\theta+t^{2n}\sin(2n-1)\theta}{1+t^2+2t\cos\theta}\Big]dt\nonumber.
\end{eqnarray}
Since   $|\Phi(\theta)|^2\in\mathbb{R}$  it follows that
\begin{eqnarray}\label{a3x234n2x34}
|\Phi(\theta)|^2=\Phi(\theta)\overline{\Phi(\theta)}=\Big[\int_0^1F(t)\Big(\frac{1}{\alpha}\nonumber\\+\frac{\cos\theta+t+t^{2n-1}\cos2n\theta+t^{2n}\cos(2n-1)\theta}{1+t^2+2t\cos\theta}\nonumber\\+i\frac{\sin\theta+t^{2n-1}\sin2n\theta+t^{2n}\sin(2n-1)\theta}{1+t^2+2t\cos\theta}\Big)dt\Big]\nonumber\\\Big[\int_0^1F(v)\Big(\frac{1}{\alpha}+\frac{\cos\theta+v+v^{2n-1}\cos2n\theta+v^{2n}\cos(2n-1)\theta}{1+v^2+2v\cos\theta}\\-i\frac{\sin\theta+v^{2n-1}\sin2n\theta+v^{2n}\sin(2n-1)\theta}{1+v^2+2v\cos\theta}\Big)dv\Big]\nonumber\\=\int_0^1\int_0^1F(t)F(v)\Big[\Big(\frac{1}{\alpha}+\frac{\cos\theta+t+t^{2n-1}\cos2n\theta+t^{2n}\cos(2n-1)\theta}{1+t^2+2t\cos\theta}\Big)\nonumber\\
\Big(\frac{1}{\alpha}+\frac{\cos\theta+v+v^{2n-1}\cos2n\theta+v^{2n}\cos(2n-1)\theta}{1+v^2+2v\cos\theta}\Big)\nonumber\\
+\Big(\frac{\sin\theta+t^{2n-1}\sin2n\theta+t^{2n}\sin(2n-1)\theta}{1+t^2+2t\cos\theta}\Big)\nonumber\\
\Big(\frac{\sin\theta+v^{2n-1}\sin2n\theta+v^{2n}\sin(2n-1)\theta}{1+v^2+2v\cos\theta}\Big)\Big].\nonumber
\end{eqnarray}
\begin{lemma}
(a)  \ Let $f,g:[0,1]\rightarrow\mathbb{R}$   be  two continuous function.  If  there  is a point $t^*\in(0,1)$  such that $f$  is decreasing on $(t^*,1),$      and the equation $g(t)=0$  has a unique root  $t_0\in[t^*, 1), $     such that   $g(t)\leq0, \ t\in[t_0,1],$    \   $g(t)\geq0, \   t\in{[0,t_0]},$    and     $f(v)\geq{f(t_0)}$  for $v\in[0,t^*],$     then we have   $$\int_0^1f(t)g(t)dt\geq{f(t_0)}\int_0^1g(t)dt.$$

(b) \ Let $f,g:[0,1]\rightarrow\mathbb{R}$  two continuous functions.   If   $f$   is a decreasing function,  and if there is a point $t_0\in(0,1)$   such that  $g(t)\geq0, \ t\in(0,t_0)$  and   $g(t)\leq0, \ t\in(t_0,1),$  then
$$\int_0^1f(t)g(t)dt\geq{f(t_0)}\int_0^1g(t)dt.$$
The statement (b)  is a particular case of (a).\\
(c) \ A well-known result is the following statement.   (Chebyshev's inequality)     If   $f$  and   $g$  are monotonic functions with  different  monotony,     then  $$\int_0^1f(t)g(t)dt\leq\int_0^1f(t)dt\int_0^1g(t)dt$$
 and in case of  the   same monotony   we have  $$\int_0^1f(t)g(t)dt\geq\int_0^1f(t)dt\int_0^1g(t)dt.$$
\end{lemma}
\begin{proof}
We have
\begin{eqnarray}\int_0^1f(t)g(t)dt=\int_0^{t_0}f(t)g(t)dt+\int_{t_0}^1f(t)g(t)dt\geq\int_0^{t_0}f(t)g(t_0)dt\nonumber\\+\int_{t_0}^1f(t)g(t_0)dt=\int_{0}^1f(t)g(t_0)dt.\nonumber\end{eqnarray}
\end{proof}
\begin{lemma}
If $\theta  \in[0,\frac{\pi}{2}],$    and     $n\geq27,$  then
\begin{eqnarray}\label{1q2} \ \ \ \ \        (a)    \ \ \ \ \ \ \  \int_0^1F(t)\Big(B(t,0)-B(t,\theta)\Big)dt \ \ \ \ \ \nonumber\\\geq\int_0^1F(t)\frac{\frac{1}{2}(1-\cos{\theta})+t^{2n-1}(1-\cos2n\theta)+t^{2n}(1-\cos(2n-1)\theta)}{(1+t)(1+t^2+2t\cos\theta)}dt,\end{eqnarray}
\begin{eqnarray}\label{z3x61q2} \ \ \ \ \ \     (b)   \ \ \ \ \ \ \ \  \int_0^1F(t)\Big(1+B(t,\theta)\Big)dt \ \ \ \ \ \nonumber\\\geq\int_0^1F(t)\frac{(1+t)(1+\cos{\theta})+t^{2n-1}(1+\cos2n\theta)+t^{2n}(1+\cos(2n-1)\theta)}{1+t^2+2t\cos\theta}dt.\end{eqnarray}
\end{lemma}
\begin{proof} According to  Lemma 2   (c),    we have
\begin{equation}\label{1234}
\int_0^1\frac{t^{2n}}{1+t}dt\leq\int_0^1{t^{2n}}dt\int_0^1\frac{1}{1+t}dt=\frac{\ln2}{2n+1}.
\end{equation}
We use  assertion  (b)   of     Lemma 2    putting $f(t)=\frac{F(t)}{1+t^2+2t\cos\theta}$   and    $g(t)=\frac{\frac{1}{2}-t-2t^{2n}}{1+t}.$    If   $\theta\in\big[0,\frac{\pi}{2}\big],$  then   the mapping $f:[0,1]\rightarrow[0,\infty)$  is strictly decreasing  and we get

\begin{eqnarray}\label{1tds0nm}\int_0^1F(t)\frac{\frac{1}{2}-t-2t^{2n}}{(1+t)(1+t^2+2t\cos\theta)}dt\geq\frac{F(t_n)}{1+t_n^2+2t_n\cos\theta}\int_0^1\Big(\frac{\frac{1}{2}-t}{1+t}\nonumber\\-2\frac{t^{2n}}{1+t}\Big)dt\geq\frac{F(t_n)}{1+t_n^2+2t_n\cos\theta}\Big(\frac{3}{2}\ln2-1-\frac{2\ln2}{2n+1}\Big)>0,
\end{eqnarray}
where $t_n$  denotes the unique root of the equation $\frac{1}{2}-t-2t^{2n}=0,$       in the interval   $ \ t\in(0,1).$
The following equality holds:
\begin{eqnarray}\label{x7v1q2}\int_0^1F(t)\Big(B(t,0)-B(t,\theta)\Big)dt \ \ \ \ =\int_0^1F(t)\frac{\frac{1}{2}(1-\cos{\theta})}{(1+t)(1+t^2+2t\cos\theta)}dt\nonumber\\+\int_0^1F(t)\frac{(t^{2n}+t^{2n-1})(1-\cos2n\theta)}{(1+t)(1+t^2+2t\cos\theta)}dt+\int_0^1F(t)\frac{(t^{2n}+t^{2n+1})(1-\cos(2n-1)\theta)}{(1+t)(1+t^2+2t\cos\theta)}dt\nonumber\\
+\int_0^1F(t)\frac{(\frac{1}{2}-t-2t^{2n})(1-\cos\theta)}{(1+t)(1+t^2+2t\cos\theta)}dt.  \ \ \ \ \end{eqnarray}
The equality   (\ref{x7v1q2})  and the inequality (\ref{1tds0nm})   imply  that
\begin{eqnarray}\label{x1mn7v1q2}\int_0^1F(t)\Big(B(t,0)-B(t,\theta)\Big)dt \ \ \ \ \geq\int_0^1F(t)\frac{\frac{1}{2}(1-\cos{\theta})}{(1+t)(1+t^2+2t\cos\theta)}dt\nonumber\\+\int_0^1F(t)\frac{t^{2n-1}(1-\cos2n\theta)}{(1+t)(1+t^2+2t\cos\theta)}dt+\int_0^1F(t)\frac{t^{2n}(1-\cos(2n-1)\theta)}{(1+t)(1+t^2+2t\cos\theta)}dt.  \ \
\end{eqnarray}
We use  assertion  (b)   of     Lemma 2    putting $f(t)=\frac{F(t)}{1+t^2+2t\cos\theta}$   and    $g(t)={t^2-2t^{2n}-t^{2n-1}}.$   If   $\theta\in[0,\frac{\pi}{2}],$   then   $f$  is strictly decreasing and we  get
\begin{eqnarray}\int_0^1F(t)\frac{t^{2}-t^{2n}-t^{2n-1}}{1+t^2+2t\cos\theta}dt\geq\frac{{F}(t_n)}{1+t_n^2+2t_n\cos\theta}\int_0^1({t^{2}-t^{2n}-t^{2n-1}})dt\nonumber\\=\frac{{F}(t_n)}{1+t_n^2+2t_n\cos\theta}\Big(\frac{1}{3}-\frac{1}{2n+1}-\frac{1}{2n}\Big)>0.\end{eqnarray}
In order to  finish the proof of  the second inequality,  we take notice of the fact that in case  $\theta\in[0,\frac{\pi}{2}]$  each member of the following sum is positive:
\begin{eqnarray}
\int_0^1F(t)\Big(1+B(t,\theta)\Big)dt= \int_0^1F(t)\frac{(1+t)(1+\cos{\theta})}{1+t^2+2t\cos\theta}dt\nonumber\\+\int_0^1F(t)\frac{t^{2n-1}(1+\cos2n\theta)}{(1+t)(1+t^2+2t\cos\theta)}dt+\int_0^1F(t)\frac{t^{2n}(1+\cos(2n-1)\theta)}{1+t^2+2t\cos\theta}dt\nonumber\\+\int_0^1F(t)\frac{t^{2}-t^{2n}-t^{2n-1}}{1+t^2+2t\cos\theta}dt+ \int_0^1F(t)\frac{t\cos\theta}{1+t^2+2t\cos\theta}dt.
\nonumber\end{eqnarray}
Thus we get
\begin{eqnarray}
\int_0^1F(t)\Big(1+B(t,\theta)\Big)dt\geq \int_0^1F(t)\frac{(1+t)(1+\cos{\theta})}{(1+t)(1+t^2+2t\cos\theta)}dt\nonumber\\+\int_0^1F(t)\frac{t^{2n-1}(1+\cos2n\theta)}{(1+t)(1+t^2+2t\cos\theta)}dt+\int_0^1F(t)\frac{t^{2n}(1+\cos(2n-1)\theta)}{(1+t)(1+t^2+2t\cos\theta)}dt,
\nonumber\end{eqnarray}
and the proof is done.
\end{proof}

\begin{lemma}
If $\theta\in\Big[\frac{\pi}{2},\frac{2\pi}{3}\Big],$      then
\begin{eqnarray}\label{jm5nnz4x1q2}\frac{5}{2}+\frac{t+\cos\theta}{1+t^2+2t\cos\theta}\geq\frac{50}{23}\frac{(1+t)(1+\cos\theta)}{1+t^2+2t\cos\theta}, \ \     (\forall)   \  t\in[0,1].\end{eqnarray}
\end{lemma}
\begin{proof}
The inequality (\ref{jm5nnz4x1q2})  is equivalent to
\begin{eqnarray}\label{gjm5s4nnz4x1q2}f(t)=\frac{5}{2}t^2+t\Big(\frac{65}{23}\cos\theta-\frac{27}{23}\Big)+\frac{15}{46}-\frac{27}{23}\cos\theta\geq0, \ \     (\forall)   \  t\in[0,1].\end{eqnarray}
The function   $f$   has a minimum point at  $t^*=\frac{65}{115}\cos\theta-\frac{27}{115}\in(0,1),$    for every     $\theta\in\Big[\frac{\pi}{2},\frac{2\pi}{3}\Big].$
Thus we get
$$f(t)\geq{f(t^*)}=\frac{15}{46}-\frac{27}{23}\cos\theta-\frac{1}{10}\Big(\frac{27-65\cos\theta}{23}\Big)^2,  \   (\forall)   \  t\in[0,1]  \  \textrm{and}    \  \theta\in\Big[\frac{\pi}{2},\frac{2\pi}{3}\Big].$$
Consequently in order to prove (\ref{jm5nnz4x1q2})  we have to show that
$$\frac{15}{46}-\frac{27}{23}\cos\theta-\frac{1}{10}\Big(\frac{27-65\cos\theta}{23}\Big)^2\geq0,  \   (\forall)   \     \  \theta\in\Big[\frac{\pi}{2},\frac{2\pi}{3}\Big].$$
Using the notation  $x=\frac{13-35\cos\theta}{35},$  we get
$$\frac{5}{12}-\frac{13}{12}\cos\theta-\frac{5}{2}\Big(\frac{13-35\cos\theta}{60}\Big)^2=-\frac{483}{2990}+\frac{27}{65}x-\frac{1}{10}x^2.$$
$$ \theta\in\Big[\frac{\pi}{2},\frac{2\pi}{3}\Big]\Rightarrow \  x\in\Big[\frac{37}{23},\frac{119}{46}\Big].$$
The  equality implies
\begin{eqnarray}\frac{5}{12}-\frac{13}{12}\cos\theta-\frac{5}{2}\Big(\frac{13-35\cos\theta}{60}\Big)^2\geq\min_{ x\in[\frac{37}{23},\frac{119}{46}\Big]}\Big\{-\frac{483}{2990}+\frac{27}{65}x-\frac{1}{10}x^2\Big\}\nonumber\\
\geq\min_{ x\in\Big[\frac{3}{2},\frac{13}{5}\Big]}\Big\{-\frac{1}{5}+\frac{5}{13}x-\frac{1}{10}x^2\Big\}=\min\Big\{g\Big(\frac{3}{2}\Big),g\Big(\frac{13}{5}\Big)\Big\}\nonumber\\=g\Big(\frac{13}{5}\Big)=\frac{31}{250}, \ \  \theta\in\Big[\frac{\pi}{2},\frac{2\pi}{3}\Big],\nonumber
\end{eqnarray}
where $g(x)=-\frac{1}{5}+\frac{5}{13}x-\frac{1}{10}x^2.$
and consequently    (\ref{jm5nnz4x1q2})  holds.
\end{proof}

\begin{lemma}
If $\theta\in\Big[\frac{\pi}{2},\frac{2\pi}{3}\Big],$    and  $n\geq27,$   then
\begin{eqnarray}\label{z4x1q2}\int_0^1F(t)\Big(B(t,0)-B(t,\theta)\Big)dt \ \ \ \ \ \nonumber\\\geq\int_0^1F(t)\frac{\frac{27}{50}(1-\cos{\theta})+t^{2n-1}(1-\cos2n\theta)+t^{2n}(1-\cos(2n-1)\theta)}{(1+t)(1+t^2+2t\cos\theta)}dt \ \  \ \ \end{eqnarray}
\begin{eqnarray}\label{q3sz3x61q2}\int_0^1F(t)\Big(2+B(t,0)+B(t,\theta)\Big)dt\geq\int_0^1F(t) \ \ \ \ \ \nonumber\\\frac{\frac{50}{27}(1+t)(1+\cos{\theta})+2t^{2n-1}(1+\cos2n\theta)+2t^{2n}(1+\cos(2n-1)\theta)}{1+t^2+2t\cos\theta}dt \ \  \end{eqnarray}
\end{lemma}

\begin{proof}
We use  assertion  (b)   of     Lemma 2,    putting $f(t)=\frac{F(t)}{(1+t)(1+t^2+2t\cos\theta)}$   and    $g(t)={\frac{27}{50}-t-2t^{2n}}.$    If   $\theta\in\big[\frac{\pi}{2},\frac{2\pi}{3}\big]$    the mapping $f:[0,1]\rightarrow[0,\infty)$  is strictly decreasing  and we get

\begin{eqnarray}\label{1td5fs0mb6nn2m}\int_0^1F(t)\frac{\frac{27}{50}-t-2t^{2n}}{(1+t)(1+t^2+2t\cos\theta)}dt\geq\frac{F(t_n)}{(1+t_n)(1+t_n^2+2t_n\cos\theta)}\int_0^1\Big({\frac{27}{50}-t}\nonumber\\-2{t^{2n}}\Big)dt>\frac{F(t_n)}{(1+t_n)(1+t_n^2+2t_n\cos\theta)}\Big(\frac{27}{50}-\frac{1}{2}-\frac{2}{2n+1}\Big)>0 \ \ \ \
\end{eqnarray}
where $t_n$  denotes the unique root of the equation $\frac{27}{50}-t-2t^{2n}=0,$       in the interval   $ \ (0,1).$\\

\begin{eqnarray}\label{2b3mx7v1q2}\int_0^1F(t)\Big(B(t,0)-B(t,\theta)\Big)dt \ \ \ \ =\int_0^1F(t)\frac{\frac{23}{50}(1-\cos{\theta})}{(1+t)(1+t^2+2t\cos\theta)}dt\nonumber\\+\int_0^1F(t)\frac{(t^{2n}+t^{2n-1})(1-\cos2n\theta)}{(1+t)(1+t^2+2t\cos\theta)}dt+\int_0^1F(t)\frac{(t^{2n}+t^{2n+1})(1-\cos(2n-1)\theta)}{(1+t)(1+t^2+2t\cos\theta)}dt\nonumber\\
+\int_0^1F(t)\frac{\big(\frac{27}{50}-t-2t^{2n}\big)(1-\cos\theta)}{(1+t)(1+t^2+2t\cos\theta)}dt v\ \ \ \  \ \  \end{eqnarray}
The equality   (\ref{2b3mx7v1q2})  and the inequality   (\ref{1td5fs0mb6nn2m})  imply (\ref{z4x1q2}).\\
In order to prove the second inequality, we remark that

\begin{eqnarray}\label{xzzxy4xvc44fg}  \int_0^1F(t)\Big(2+B(t,0)+B(t,\theta)\Big)dt= \int_0^1F(t)\Big(\frac{5}{2}+\frac{1+t^{2n-1}}{1+t}-\frac{1}{2}+B(t,\theta)\Big)dt\nonumber\\
= \int_0^1F(t)\Big(\frac{5}{2}+\frac{1+2t^{2n-1}-t}{2(1+t)}+\frac{t+\cos\theta+t^{2n}\cos(2n-1)\theta+t^{2n-1}\cos(2n\theta)}{1+t^2+2t\cos\theta}\Big)dt \ \ \ \ \\
= \int_0^1F(t)\Big(\frac{5}{2}+\frac{t+\cos\theta}{1+t^2+2t\cos\theta}+\frac{1+2t^{2n-1}-t}{2(1+t)}+\frac{t^{2n}\cos(2n-1)\theta+t^{2n-1}\cos(2n\theta)}{1+t^2+2t\cos\theta}\Big)dt\nonumber\\
= \int_0^1F(t)\Big(\frac{5}{2}+\frac{t+\cos\theta}{1+t^2+2t\cos\theta}+\frac{2t^{2n}(1+\cos(2n-1)\theta)+2t^{2n-1}(1+\cos(2n\theta))}{1+t^2+2t\cos\theta}\Big)dt\nonumber\\
+ \int_0^1F(t)\frac{2t^{2n-1}}{1+t}dt+ \int_0^1F(t)\Big(\frac{1-t}{2(1+t)}-\frac{t^{2n}(2+\cos(2n-1)\theta)+t^{2n-1}(2+\cos(2n\theta))}{1+t^2+2t\cos\theta}\Big)dt\nonumber
\end{eqnarray}

  We put  $g_1(t)=1-2t+2t^2-t^3-24t^{2n-1}$        $g_2(t)=2t^3-8t^{2n-1},$       and     $f(t)=\frac{F(t)}{2(1+t^3)}$  in the assertion  (b)   of Lemma 2,   and we get
\begin{eqnarray}\label{d2s1xcc4b3fm12nijf}
 \int_0^1F(t)\Big(\frac{1-t}{2(1+t)}-\frac{t^{2n}(2+\cos(2n-1)\theta)+t^{2n-1}(2+\cos(2n\theta))}{1+t^2+2t\cos\theta}\Big)dt\nonumber\\
\geq \int_0^1F(t)\Big(\frac{1-t}{2(1+t)}-\frac{3t^{2n}+3t^{2n-1}}{1+t^2+2t\cos\theta}\Big)dt\geq \int_0^1F(t)\Big(\frac{1-t}{2(1+t)}\\-\frac{6t^{2n-1}}{1+t^2-t}\Big)dt 
 \geq \int_0^1\frac{F(t)}{2(1+t^3)}\Big(1-2t+2t^2-t^3-12t^{2n-1}(1+t)\Big)dt \nonumber\\
\geq  \int_0^1\frac{F(t)}{2(1+t^3)}\Big(1-2t+2t^2-t^3-24t^{2n-1}\Big)dt \nonumber\\
  \frac{F(t^*)}{2(1+(t^*)^3)}    \int_0^1\Big(1-2t+2t^2-t^3-24t^{2n-1}\Big)dt\nonumber\\ = \frac{F(t^*)}{2(1+(t^*)^3)} \Big(\frac{5}{12}-\frac{12}{n} \Big)>0,     \ \    \theta\in\Big[\frac{\pi}{2},\frac{2\pi}{3}\Big].\nonumber
\end{eqnarray}
Finally equality (\ref{xzzxy4xvc44fg}),   Lemma 4    and  inequality  (\ref{d2s1xcc4b3fm12nijf})    imply   (\ref{q3sz3x61q2}),   and the proof is  done.
\end{proof}

\section{The Main Result}

\begin{theorem}
If $n$  is a natural number,  $n\geq52$     and    $\alpha\in(0,1)$   then the following inequality holds
\begin{equation}\label{us5n6sw}
|A_{2n-1}(\alpha,e^{i\theta})|\leq|A_{2n-1}(\alpha,1)|, \ \  \textrm{for  \  all} \ \theta\in[-\frac{\pi}{2},\frac{\pi}{2}].
\end{equation}
\end{theorem}
\begin{proof}
According to  (\ref{234n234})  and   (\ref{a3x234n2x34})    the inequality (\ref{us5n6sw})  is equivalent to  \begin{equation}\label{lfd8sn}|\Phi(0)|^2\geq|\Phi(\theta)|^2,   \ \theta\in[-\pi,\pi]\Leftrightarrow  |\Phi(0)|^2-|\Phi(\theta)|^2\geq0,   \ \theta\in[-\frac{\pi}{2},\frac{\pi}{2}].\end{equation}
We denote $$B(t,\theta)=\frac{\cos\theta+t+t^{2n-1}\cos2n\theta+t^{2n}\cos(2n-1)\theta}{1+t^2+2t\cos\theta},$$
 $$C(t,\theta)=\frac{\sin\theta+t^{2n-1}\sin2n\theta+t^{2n}\sin(2n-1)\theta}{1+t^2+2t\cos\theta}.$$
The equality   (\ref{a3x234n2x34})  implies that
\begin{eqnarray}\label{b5x3yo04d6z}
\Phi^2(0)-|\Phi(\theta)|^2=\int_0^1\int_0^1F(t)F(v)\Big[\Big(\frac{1}{\alpha}+B(t,0)\Big)\Big(\frac{1}{\alpha}+B(v,0)\Big) \\-\Big(\frac{1}{\alpha}+B(t,\theta)\Big)\Big(\frac{1}{\alpha}+B(v,\theta)\Big)-C(t,\theta)C(v,\theta)\Big]dtdv\nonumber\\
=\int_0^1\int_0^1F(t)F(v)\Big[\frac{1}{\alpha}\Big(B(t,0)-B(t,\theta)\Big)+\frac{1}{\alpha}\Big(B(v,0)-B(v,\theta)\Big)\nonumber\\+B(t,0)B(v,0)-B(t,\theta)B(v,\theta)-C(t,\theta)C(v,\theta)\Big]dtdv.\nonumber
\end{eqnarray}
It is easily seen that  $\Phi^2(0)-|\Phi(\theta)|^2$ is an even function with respect to $\theta,$ consequently we have to prove (\ref{lfd8sn})  only for  $\theta\in[0,\frac{\pi}{2}].$\\
Lemma 3,  (a)  implies that in case    $\theta\in[0,\frac{\pi}{2}],$  the following  inequality holds     $\int_0^1F(t)\Big(B(t,0)-{B}(t,\theta)\Big)dt\geq0.$
Thus we infer that
\begin{eqnarray}\label{f5b5xy4d6z}
\Phi^2(0)-|\Phi(\theta)|^2=\int_0^1\int_0^1F(t)F(v)\Big[\Big(B(t,0)-B(t,\theta)\Big)\Big(1+B(v,\theta)\Big)\nonumber\\+\Big(B(v,0)-B(v,\theta)\Big)\Big(1+B(t,0)\Big)\Big]dtdv-\int_0^1\int_0^1F(t)F(v)C(t,\theta)C(v,\theta)dtdv\nonumber\\
+\int_0^1\int_0^1F(t)F(v)\Big(B(t,0)-B(t,\theta)\Big)\Big(B(v,0)-B(v,\theta)\Big)dtdv\nonumber\\=\int_0^1\int_0^1F(t)F(v)\Big[\Big(B(t,0)-B(t,\theta)\Big)\Big(1+B(v,\theta)\Big)\nonumber\\+\Big(B(v,0)-B(v,\theta)\Big)\Big(1+B(t,0)\Big)\Big]dtdv-\int_0^1\int_0^1F(t)F(v)C(t,\theta)C(v,\theta)dtdv\nonumber\\
+\int_0^1F(t)\Big(B(t,0)-B(t,\theta)\Big)dt\int_0^1F(v)\Big(B(v,0)-B(v,\theta)\Big)dv\nonumber\\\nonumber\\\geq\int_0^1\int_0^1F(t)F(v)\Big[\Big(B(t,0)-B(t,\theta)\Big)\Big(1+B(v,\theta)\Big)\nonumber\\+\Big(B(v,0)-B(v,\theta)\Big)\Big(1+B(t,0)\Big)\Big]dtdv-\int_0^1\int_0^1F(t)F(v)C(t,\theta)C(v,\theta)dtdv\nonumber
\end{eqnarray}
This inequality  is equivalent to
\begin{eqnarray}\label{w1w2c4x2zz9}
\Phi^2(0)-|\Phi(\theta)|^2\geq\int_0^1F(t)\Big(B(t,0)-B(t,\theta)\Big)dt\int_0^1F(v)\Big(1+B(v,\theta)\Big)dv\nonumber\\+
\int_0^1F(t)\Big(1+B(t,\theta)\Big)dt\int_0^1F(v)\Big(B(v,0)-B(v,\theta)\Big)dv \ \ \ \ \ \ \\-\int_0^1\int_0^1F(t)F(v)C(t,\theta)C(v,\theta)dtdv \ \ \ \ \ \  \nonumber
\end{eqnarray}
The inequality between the arithmetic and geometric  means  leads to
\begin{eqnarray}\label{111w1w2c4x2zz9}
\int_0^1F(t)\Big(B(t,0)-B(t,\theta)\Big)dt\int_0^1F(v)\Big(1+B(v,\theta)\Big)dv\nonumber\\+
\int_0^1F(t)\Big(1+B(t,\theta)\Big)dt\int_0^1F(v)\Big(B(v,0)-B(v,\theta)\Big)dv \ \ \ \ \ \ \\-\int_0^1\int_0^1F(t)F(v)C(t,\theta)C(v,\theta)dtdv \ \ \ \ \ \  \nonumber\\
\geq2\Big[\int_0^1F(t)\Big(B(t,0)-B(t,\theta)\Big)dt\int_0^1F(v)\Big(1+B(v,\theta)\Big)dv\nonumber\\
\int_0^1F(t)\Big(1+B(t,\theta)\Big)dt\int_0^1F(v)\Big(B(v,0)-B(v,\theta)\Big)dv\Big]^{\frac{1}{2}} \nonumber \\-\int_0^1\int_0^1F(t)F(v)C(t,\theta)C(v,\theta)dtdv\nonumber\\
=2\int_0^1F(t)\Big(1+B(t,\theta)\Big)dt\int_0^1F(t)\Big(B(t,0)-B(t,\theta)\Big)dt \nonumber \\-\Big(\int_0^1F(t)C(t,\theta)dt\Big)^2\nonumber
\end{eqnarray}
The Cauchy-Schwarz  inequality for integrals implies
\begin{eqnarray}\label{wndlkoop}
2\int_0^1F(t)\Big(1+B(t,\theta)\Big)dt\int_0^1F(t)\Big(B(t,0)-B(t,\theta)\Big)dt \ \ \ \ \ \  \nonumber \\-\Big(\int_0^1F(t)C(t,\theta)dt\Big)^2
\geq2\Big\{\int_0^1F(t)\Big[\Big(1+B(t,\theta)\Big)\Big(B(t,0)-B(t,\theta)\Big)\Big]^{\frac{1}{2}}dt\Big\}^2 \ \ \ \ \\
-\Big(\int_0^1F(t)C(t,\theta)dt\Big)^2 \ \ \ \ \ \nonumber
\end{eqnarray}
Lemma 3 implies
\begin{eqnarray}\label{ed55f}
2\Big\{\int_0^1F(t)\Big[\Big(1+B(t,\theta)\Big)\Big(B(t,0)-B(t,\theta)\Big)\Big]^{\frac{1}{2}}dt\Big\}^2 \ \ \ \ \\
-\Big(\int_0^1F(t)C(t,\theta)dt\Big)^2\geq\nonumber\\
2\Big(\int_0^1F(t)\sqrt{\frac{\frac{1}{2}(1-\cos{\theta})+t^{2n-1}(1-\cos2n\theta)+t^{2n}(1-\cos(2n-1)\theta)}{(1+t)(1+t^2+2t\cos\theta)}}\nonumber\\
\sqrt{\frac{(1+t)(1+\cos{\theta})+t^{2n-1}(1+\cos2n\theta)+t^{2n}(1+\cos(2n-1)\theta)}{1+t^2+2t\cos\theta}}dt\Big)^2\nonumber\\
-\Big(\int_0^1F(t)C(t,\theta)dt\Big)^2.\nonumber
\end{eqnarray}

Putting $a_1^2=\frac{\frac{1}{2}(1-\cos{\theta})}{(1+v)(1+v^2+2v\cos\theta)},$    \  $b_1^2=\frac{(1+t)(1+\cos{\theta})}{1+v^2+2v\cos\theta}$  and so an, in the inequality
$$\sqrt{a_1^2+a_2^2+a_3^2}\sqrt{b_1^2+b_2^2+b_3^2}\geq{|a_1b_1|+|a_2b_2|+|a_3b_3|},$$
we get
\begin{eqnarray}\label{xcs3bnm3xcxc2}
2\Big(\int_0^1F(t)\sqrt{\frac{\frac{1}{2}(1-\cos{\theta})+t^{2n-1}(1-\cos2n\theta)+t^{2n}(1-\cos(2n-1)\theta)}{(1+t)(1+t^2+2t\cos\theta)}}\nonumber\\
\sqrt{\frac{(1+t)(1+\cos{\theta})+t^{2n-1}(1+\cos2n\theta)+t^{2n}(1+\cos(2n-1)\theta)}{1+t^2+2t\cos\theta}}dt\Big)^2\\
-\Big(\int_0^1F(t)C(t,\theta)dt\Big)^2\geq\nonumber\\
2\Big(\int_0^1F(t)\frac{\sqrt{\frac{1}{2}(1+t)}|\sin\theta|+t^{2n-1}|\sin2n\theta|+t^{2n}|\sin(2n-1)\theta|}{(1+t^2+2t\cos\theta)\sqrt{1+t}}
dt\Big)^2\nonumber\\
-\Big(\int_0^1F(t)C(t,\theta)dt\Big)^2.\nonumber
\end{eqnarray}
On the other hand we have
\begin{eqnarray}\label{r5txcvcckxcxc2}
2\Big(\int_0^1F(t)\frac{\sqrt{\frac{1}{2}(1+t)}|\sin\theta|+t^{2n-1}|\sin2n\theta|+t^{2n}|\sin(2n-1)\theta|}{(1+t^2+2t\cos\theta)\sqrt{1+t}}
dt\Big)^2\\
-\Big(\int_0^1F(t)C(t,\theta)dt\Big)^2\geq\nonumber\\
\Big(\int_0^1F(t)\frac{\sqrt{(1+t)}|\sin\theta|+\sqrt{2}t^{2n-1}|\sin2n\theta|+\sqrt{2}t^{2n}|\sin(2n-1)\theta|}{(1+t^2+2t\cos\theta)\sqrt{1+t}}
dt\Big)^2\nonumber\\
-\Big(\int_0^1F(t)C(t,\theta)dt\Big)^2\geq\nonumber\\
\Big(\int_0^1F(t)\frac{|\sin\theta|+t^{2n-1}|\sin2n\theta|+t^{2n}|\sin(2n-1)\theta|}{1+t^2+2t\cos\theta}
dt\Big)^2\nonumber\\
-\Big(\int_0^1F(t)\frac{\sin\theta+t^{2n-1}\sin2n\theta+t^{2n}\sin(2n-1)\theta}{1+t^2+2t\cos\theta}dt\Big)^2\geq0, \ \ \theta\in[0,\frac{\pi}{2}],\nonumber
\end{eqnarray}
Finally the inequalities  (\ref{w1w2c4x2zz9}), (\ref{111w1w2c4x2zz9}), (\ref{wndlkoop}), (\ref{ed55f}),  (\ref{xcs3bnm3xcxc2})  and   (\ref{r5txcvcckxcxc2})  imply that
$$
\Phi^2(0)-|\Phi(\theta)|^2\geq \ \theta\in[0,\frac{\pi}{2}],\nonumber
$$ and consequently  inequalities   (\ref{lfd8sn})      and     (\ref{us5n6sw})     hold in case $\theta\in[-\frac{\pi}{2},\frac{\pi}{2}].$\\
\end{proof}
\begin{theorem}
If $n$  is a natural number,  $n\geq51$     and    $\alpha\in(0,1),$   then the following inequality holds
\begin{equation}\label{usxzs5n6sw}
|A_{2n-1}(\alpha,e^{i\theta})|\leq|A_{2n-1}(\alpha,1)|, \ \  \textrm{for  \  all} \ \theta\in[-\frac{2\pi}{3},-\frac{\pi}{2}]\cup[\frac{\pi}{2},\frac{2\pi}{3}]
\end{equation}
\end{theorem}
\begin{proof}
Equality   (\ref{b5x3yo04d6z})  can be rewritten as follows
\begin{eqnarray}\label{d0c0bb}2\Big(\Phi^2(0)-|\Phi(\theta)|^2\Big) \ \ \nonumber\\=\int_0^1F(t)\Big(B(t,0)-B(t,\theta)\Big)dt\int_0^1F(v)\Big(\frac{2}{\alpha}+B(v,0) +B(v,\theta)\Big)dv \ \  \ \\+\int_0^1F(v)\Big(B(v,0)-B(v,\theta)\Big)dv\int_0^1F(t)\Big(\frac{2}{\alpha}+B(t,0)+B(t,\theta)\Big)\Big]dt \ \ \  \nonumber\\-2\int_0^1\int_0^1F(t)F(v)C(t,\theta)C(v,\theta)dtdv. \ \ \ \nonumber
\end{eqnarray}
Since     $2\Big(\Phi^2(0)-|\Phi(\theta)|^2\Big)$  defines an even function  in order to prove   (\ref{usxzs5n6sw})   it is enough to show that
\begin{equation}
\Phi^2(0)-|\Phi(\theta)|^2\geq0, \ \ \theta\in\Big[\frac{\pi}{2},\frac{2\pi}{3}\Big].
\end{equation}
The inequality between the arithmetic and geometric  means   and the condition   $\alpha\in(0,1]$   imply
\begin{eqnarray}\label{kx13eed0sc0bb}2\Big(\Phi^2(0)-|\Phi(\theta)|^2\Big)\geq2\Big[\int_0^1F(t)\Big(B(t,0)-B(t,\theta)\Big)dt\int_0^1F(v)\Big(2+B(v,0) \ \ \ \ \ \ \ \ \ \\ +B(v,\theta)\Big)dv\int_0^1F(v)\Big(B(v,0)-B(v,\theta)\Big)dv\int_0^1F(t)\Big(2+B(t,0)+B(t,\theta)\Big)dt\Big]^{\frac{1}{2}} \ \ \ \ \ \ \nonumber\\-2\int_0^1\int_0^1F(t)F(v)C(t,\theta)C(v,\theta)dtdv \ \ \ \ \ \ \ \ \ \ \ \ \ \ \  \nonumber
\nonumber\\=2\int_0^1F(t)\Big(B(t,0)-B(t,\theta)\Big)dt\int_0^1F(t)\Big(2+B(t,0)+B(t,\theta\Big)dt \ \ \ \ \ \  \ \ \ \ \nonumber\\-2\Big(\int_0^1F(t)C(t,\theta)dt\Big)^2 \ \ \ \ \ \ \ \ \nonumber\end{eqnarray}
Lemma 5 and  inequality  (\ref{kx13eed0sc0bb})  lead to
\begin{eqnarray}\label{yhnx13eed0sc0bb}\Phi^2(0)-|\Phi(\theta)|^2\nonumber\\\geq\int_0^1F(t)\frac{\frac{12}{25}(1-\cos{\theta})+t^{2n-1}(1-\cos2n\theta)+t^{2n}(1-\cos(2n-1)\theta)}{(1+t)(1+t^2+2t\cos\theta)}dt \ \ \ \\
\int_0^1F(t) \frac{\frac{25}{12}(1+t)(1+\cos{\theta})+2t^{2n-1}(1+\cos2n\theta)+2t^{2n}(1+\cos(2n-1)\theta)}{1+t^2+2t\cos\theta}dt\nonumber\\-\Big(\int_0^1F(t)C(t,\theta)dt\Big)^2\nonumber\end{eqnarray}
We apply twice the  Cauchy-Schwarz inequality    jest like  in the proof of the previous theorem  and we get that  in order to prove  (\ref{yhnx13eed0sc0bb}) it is enough to show that
\begin{eqnarray}\label{yhnx13evcxed0sc0bb}\Phi^2(0)-|\Phi(\theta)|^2\nonumber\\\geq\Big(\int_0^1F(t)\frac{\sqrt{1+t}|\sin{\theta}|+\sqrt{2}t^{2n-1}|\sin{2n}\theta|+\sqrt{2}t^{2n}|\sin({2n-1})\theta|}{(1+t^2+2t\cos\theta)\sqrt{1+t}}dt\Big)^2\nonumber\\ -\Big(\int_0^1F(t)\frac{\sin\theta+t^{2n-1}\sin2n\theta+t^{2n}\sin(2n-1)\theta}{1+t^2+2t\cos\theta}dt\Big)^2\nonumber\end{eqnarray}
This inequality is equivalent to
\begin{eqnarray}\label{xcxcxc2}
\Phi^2(0)-|\Phi(\theta)|^2\geq\nonumber\\
\Big(\int_0^1F(t)\frac{|\sin\theta|}{1+t^2+2t\cos\theta}dt+\int_0^1F(t)\frac{\sqrt{2}t^{2n-1}|\sin2n\theta|}{(1+t^2+2t\cos\theta)\sqrt{1+t}}dt \ \ \\+\int_0^1F(t)\frac{\sqrt{2}t^{2n}|\sin(2n-1)\theta|}{(1+t^2+2t\cos\theta)\sqrt{1+t}}
dt\Big)^2\nonumber\\
-\Big(\int_0^1F(t)\frac{\sin\theta}{1+t^2+2t\cos\theta}dt+\int_0^1F(t)\frac{t^{2n-1}\sin2n\theta}{(1+t^2+2t\cos\theta)}dt\nonumber\\+\int_0^1F(t)\frac{t^{2n}\sin(2n-1)\theta}{(1+t^2+2t\cos\theta)\sqrt{1+t}}
dt\Big)^2\geq0, \ \   \theta\in\Big[\frac{\pi}{2},\frac{2\pi}{3}\Big].\nonumber
\end{eqnarray}
and the proof is done.
\end{proof}
\begin{theorem}
If $n$  is a natural number,  $n\geq27,$     and    $\alpha\in[\frac{1}{3},1),$   then the following inequality holds
\begin{equation}\label{uscsaxzs5n6sw}
{A_{2n-1}(\alpha,1)}\geq|A_{2n-1}(\alpha,e^{i\theta})|, \ \  \textrm{for  \  all} \ \theta\in[-\pi,-\frac{2\pi}{3}]\cup[\frac{2\pi}{3},\pi]
\end{equation}
\end{theorem}
\begin{proof}
We  will use the Taylor formula   with an integral remainder.  Let $f:(-a,a)\rightarrow\mathbb{R}$  be  a  $2n$  times derivabile function, such that $f^{(2n)}$  is continous. If   $x\in(-a,a),$  then
\begin{eqnarray}f(x)=f(0)+\frac{f'(0)}{1!}x+\frac{f''(0)}{2!}x^2+\ldots+\frac{f^{(2n-1)}(0)}{(2n-1)!}x^{2n-1}\nonumber\\+\frac{x^{2n}}{(2n-1)!}\int_0^1(1-t)^{2n-1}f^{(2n)}(xt)dt.\nonumber\end{eqnarray}
Let $f$   be the function defined by
$f:(-1,1)\rightarrow\mathbb{R}, \ \ f(x)=(1+x)^\alpha, \ \
\alpha\in(0,1)$   and  we get
\begin{eqnarray}(1+x)^\alpha=1+\frac{\alpha}{1!}x+\frac{\alpha(\alpha-1)}{2!}x^2+\ldots+\frac{\alpha(\alpha-1)\ldots(\alpha-2n+2)}{(2n-1)!}x^{2n-1}\nonumber\\
+x^{2n}\frac{\alpha(\alpha-1)\ldots(\alpha-2n+1)}{(2n-1)!}\int_0^1(1-t)^{2n-1}(1+xt)^{\alpha-2n}dt.\nonumber\end{eqnarray}
Since the mapping  $f:U\rightarrow\mathbb{R}, \ \
f(z)=(1+z)^\alpha$ is well defined(we take the principal branch of
the multi valued function)  it follows that the equality
\begin{eqnarray}(1+z)^\alpha=1+\frac{\alpha}{1!}z+\frac{\alpha(\alpha-1)}{2!}z^2+\ldots+\frac{\alpha(\alpha-1)\ldots(\alpha-2n+2)}{(2n-1)!}z^{2n-1}\nonumber\\
+z^{2n}\frac{\alpha(\alpha-1)\ldots(\alpha-2n+1)}{(2n-1)!}\int_0^1(1-t)^{2n-1}(1+zt)^{\alpha-2n}dt.\nonumber\end{eqnarray}
holds for every $z\in{U}.$ The mapping  $f:U\rightarrow\mathbb{R},
\ \ f(z)=(1+z)^\alpha$  is radially continuous and so we infer
that
\begin{eqnarray}(1+e^{i\theta})^\alpha=1+\frac{\alpha}{1!}e^{i\theta}+\frac{\alpha(\alpha-1)}{2!}e^{2i\theta}+\ldots+\frac{\alpha(\alpha-1)\ldots(\alpha-2n+2)}{(2n-1)!}e^{(2n-1)i\theta}\nonumber\\
+e^{2ni\theta}\frac{\alpha(\alpha-1)\ldots(\alpha-2n+1)}{(2n-1)!}\int_0^1(1-t)^{2n-1}(1+e^{i\theta}t)^{\alpha-2n}dt,
\ \theta\in(-\pi,\pi).\nonumber\end{eqnarray} Taking the absolute
value, this equality implies  that
\begin{eqnarray}\label{fkfk354}|A_{2n-1}(\alpha,e^{i\theta})|\nonumber\\=\big|1+\frac{\alpha}{1!}e^{i\theta}+\frac{\alpha(\alpha-1)}{2!}e^{2i\theta}+\ldots
+\frac{\alpha(\alpha-1)\ldots(\alpha-2n+2)}{(2n-1)!}e^{(2n-1)i\theta}\big|\nonumber\\
\leq\frac{\alpha(1-\alpha)(2-\alpha)\ldots(2n-1-\alpha)}{(2n-1)!}\int_0^1(1-t)^{2n-1}\big|1+e^{i\theta}t\big|^{\alpha-2n}dt\nonumber\\+\big|1+e^{i\theta}
\big|^\alpha\leq\big|1+e^{i\theta}
\big|^\alpha++\alpha(1-\alpha)(1-\frac{\alpha}{2})(1-\frac{\alpha}{3})\ldots(1-\ \ \nonumber\\\frac{\alpha}{2n-1})\int_0^1(1-t)^{2n-1}\big|1+e^{i\theta}t\big|^{\alpha-2n}dt,
 \ \theta\in(-\pi,\pi).\nonumber\end{eqnarray}
 On the other hand we have
 \begin{eqnarray}\label{ramglp67}A_{2n-1}(\alpha,1)=1+\frac{\alpha}{1!}-\frac{\alpha(1-\alpha)}{2!}+\frac{\alpha(1-\alpha)(2-\alpha)}{3!}\nonumber\\
 -\frac{\alpha(1-\alpha)(2-\alpha)(3-\alpha)}{4!}+\ldots
+\frac{\alpha(1-\alpha)\ldots(2n-2-\alpha)}{(2n-1)!}\geq1+\frac{\alpha(1+\alpha)}{2}\nonumber\end{eqnarray}
 Thus in order to prove   (\ref{uscsaxzs5n6sw})   we have to show
 that the following inequality  holds
 \begin{eqnarray}\label{ttsjjkk}
1+\frac{\alpha(1+\alpha)}{2}\geq\big|1+e^{i\theta}
\big|^\alpha++\alpha(1-\alpha)(1-\frac{\alpha}{2})(1-\frac{\alpha}{3})\ldots(1-\ \ \nonumber\\\frac{\alpha}{2n-1})\int_0^1(1-t)^{2n-1}\big|1+e^{i\theta}t\big|^{\alpha-2n}dt,
 \  \theta\in[\frac{2\pi}{3},\pi].\nonumber
 \end{eqnarray}
 We denote $x=-\cos\theta,$  and the inequality (\ref{ttsjjkk})   will be
 equivalent to
\begin{eqnarray}\label{t3tsjvjkk}
1+\frac{\alpha(1+\alpha)}{2}\geq\big(2-2x\big)^\frac{\alpha}{2}+\alpha(1-\alpha)(1-\frac{\alpha}{2})(1-\frac{\alpha}{3})\ldots(1-\ \ \nonumber\\\frac{\alpha}{2n-1})\int_0^1\Big(\frac{1-t}{\sqrt{1+t^2-2tx}}\Big)^{2n-1}\big(\sqrt{1+t^2-2tx}\big)^{\alpha-1}dt, 
  \ x\in[\frac{1}{2},1],
 \end{eqnarray}
  and this inequality can be rewritten in the following form:
  \begin{eqnarray}\label{t3ts1jwv6zpx42jk2k}
1+\alpha\geq\frac{\big(2-2x\big)^\frac{\alpha}{2}-1}{\frac{\alpha}{2}}+2(1-\alpha)(1-\frac{\alpha}{2})(1-\frac{\alpha}{3})\ldots(1-\ \ \nonumber\\\frac{\alpha}{2n-1})\int_0^1\Big(\frac{1-t}{\sqrt{1+t^2-2tx}}\Big)^{2n-1}\big(\sqrt{1+t^2-2tx}\big)^{\alpha-1}dt,  \ \ \
 \ \ \ x\in[\frac{1}{2},1].\ \ \ 
 \end{eqnarray}
It is easily seen that    $$2(1-\alpha)(1-\frac{\alpha}{2})(1-\frac{\alpha}{3})\ldots(1-\ \ \nonumber\\\frac{\alpha}{2n-1})\int_0^1\Big(\frac{1-t}{\sqrt{1+t^2-2tx}}\Big)^{2n-1}\big(\sqrt{1+t^2-2tx}\big)^{\alpha-1}dt$$  is decreasing with respect to $n$    and  $x,$   and  $\frac{\big(2-2x\big)^\frac{\alpha}{2}-1}{\frac{\alpha}{2}}<0,$  for all  $\alpha\in(0,1),$     and   $ x\in[\frac{1}{2},1].$   
Thus in order to prove   (\ref{t3ts1jwv6zpx42jk2k})  it is enough to prove the inequality 
\begin{eqnarray}\label{h6xnlnp0xjj1jv6zx4jk2k}
1+\frac{1}{3}\geq2(1-\alpha)(1-\frac{\alpha}{2})(1-\frac{\alpha}{3})\ldots(1-\frac{\alpha}{53})\int_0^1\big(1-t\big)^{\alpha-1}dt
 \end{eqnarray}
in case $\alpha=1,$   that is 
$$\frac{4}{3}\geq6(1-\frac{1}{3})(1-\frac{1}{3\cdot2})(1-\frac{1}{3\cdot3})\ldots(1-\frac{1}{3\cdot27}).$$
This inequality holds and the proof  is done.
\end{proof}
\section{Concluding Remarcs}
The following two corollaries are the proof of the Brannan conjecture in two different particular cases.\\
Theorem 1,    Theorem  2    and the result of  \cite{6}    imply the following corollary.
\begin{corollary}
If  $x\in\mathbb{C}$     with  $|\arg{x}|\leq\frac{2\pi}{3},$   and  $|x|=1,$   then the inequality
$$|A_{2n-1}(\alpha,x)|\leq{A_{2n-1}}(\alpha,1)$$
holds for every  $\alpha\in(0,1).$
\end{corollary}
Theorem 1,  Theorem  2,   Theorem 3    and the result of  \cite{6}   imply the following corollary.  This corollary is the solution of the Brannan conjecture in case   $\alpha \in [\frac{1}{2},1).$
\begin{corollary}
The inequality
$$|A_{2n-1}(\alpha,x)|\leq{A_{2n-1}}(\alpha,1)$$
holds for every  $\alpha\in[\frac{1}{3},1),$   and  $x\in\mathbb{C},$   with  $|x|=1.$
\end{corollary}
\begin{conjecture}
Numerical resuls suggest that the inequality 
\begin{eqnarray}\label{s5d6f7t3tsjvjkk}
1+\frac{\alpha(1+\alpha)}{2}\geq\big(2-2x\big)^\frac{\alpha}{2}+\alpha(1-\alpha)(1-\frac{\alpha}{2})(1-\frac{\alpha}{3})\ldots(1-\ \ \nonumber\\\frac{\alpha}{2n-1})\int_0^1\Big(\frac{1-t}{\sqrt{1+t^2-2tx}}\Big)^{2n-1}\big(\sqrt{1+t^2-2tx}\big)^{\alpha-1}dt, 
  \ x\in[\frac{1}{2},1],
 \end{eqnarray}
holds for every $\alpha\in(0,\frac{1}{3}).$
\end{conjecture}
\begin{remark}
If the previous conjecture holds, then the conjecture of Brannan holds  in case $\beta=1$   and every $\alpha\in(0,1).$
\end{remark}

e-mail address: rszasz@ms.sapientia.ro

\end{document}